\documentclass[12pt, leqno]{amsart}
\usepackage{graphics,amsmath,amsthm,amssymb,graphicx}
\usepackage[enableskew]{youngtab}
\usepackage{tikz}
\usepackage[all,cmtip]{xy}
\usepackage{color}
\makeatletter
\newsavebox\myboxA
\newsavebox\myboxB
\newlength\mylenA

\newcommand\encircle[1]{%
  \tikz[baseline=(X.base)] 
    \node (X) [draw, shape=circle, inner sep=0] {\strut #1};}

\newcommand*\xoverline[2][0.75]{%
    \sbox{\myboxA}{$\m@th#2$}%
    \setbox\myboxB\null
    \ht\myboxB=\ht\myboxA%
    \dp\myboxB=\dp\myboxA%
    \wd\myboxB=#1\wd\myboxA
    \sbox\myboxB{$\m@th\overline{\copy\myboxB}$}
    \setlength\mylenA{\the\wd\myboxA}
    \addtolength\mylenA{-\the\wd\myboxB}%
    \ifdim\wd\myboxB<\wd\myboxA%
       \rlap{\hskip 0.5\mylenA\usebox\myboxB}{\usebox\myboxA}%
    \else
        \hskip -0.5\mylenA\rlap{\usebox\myboxA}{\hskip 0.5\mylenA\usebox\myboxB}%
    \fi}
\makeatother
\newtheorem{theorem}{Theorem}[section]

\newtheorem{proposition}[theorem]{Proposition}
\newtheorem{corollary}[theorem]{Corollary}
\newtheorem{definition}[theorem]{Definition}

\newtheorem{lemma}[theorem]{Lemma}
\newtheorem{example}[theorem]{Example}

\newtheorem{conjecture}[theorem]{Conjecture}

\begin{document}

\title[Abaci structures of $(s,ms\pm1)$-core partitions]{Abaci structures of $(s, ms\pm1)$-core partitions}
\author[R. Nath]{Rishi Nath}
\address{Department of Mathematics} 
\address {York College, City University of New York, Jamaica, NY 11451}
\address{rnath@york.cuny.edu}
\author[J. A. Sellers]{James A. Sellers}
\address{Department of Mathematics, Penn State University, University Park, PA  16802}
\address {sellersj@psu.edu}


\date{\bf\today}

\begin{abstract}
We develop a geometric approach to the study of $(s,ms-1)$-core and $(s,ms+1)$-core partitions through the associated $ms$-abaci. This perspective yields new proofs for results of H. Xiong and A. Straub (originally proposed by T. Amdeberhan) on the enumeration of $(s, s+1)$ and $(s,ms-1)$-core partitions with distinct parts. It also enumerates the $(s, ms+1)$-cores with distinct parts. Furthermore, we calculate the weight of the $(s, ms-1,ms+1)$-core partition with the largest number of parts.  Finally we use 2-core partitions to enumerate self-conjugate core partitions with distinct parts. The central idea is that the $ms$-abaci of maximal $(s,ms\pm1)$-cores can be built up from $s$-abaci of $(s,s\pm 1)$-cores in an elegant way. 
\end{abstract}

\maketitle


\bigskip

\section{Introduction}
\subsection{Partitions and abacus diagrams}
A {\it partition} $\lambda$ of the positive integer $n$ is a weakly decreasing sequence of positive integers which sum to $n.$  We will call $n$ the weight of $\lambda$. Each of the integers which make up the partition is known as a {\it part} of the partition.  For example, $(8,6,5,5,3,2,2,2,1)$ is a partition with weight $n=34$, and is alternatively written as $(8,6,5^2,3,2^3,1).$  

A {\it Young diagram} is a pictorial representation of a partition.  Simply put, it is a finite collection of boxes which are arranged in left-justified rows with the row lengths weakly decreasing (since each row of the Young diagram corresponds to a part in the partition).  
To each box in the Young diagram of $\lambda$ we assign a {\it hook}, which is the set of boxes in the same row and to the right, and in the same column and below, as well as the box itself, which is called the {\it corner} of the hook.  We use matrix notation to label the hooks: $h_{ij}$ is the hook whose corner is in the $i$-th row and the $j$-th column. The number of boxes $|h_{ij}|$ is the {\it hook length} of $h_{ij}$. The {\it first-column hook lengths} are those that appear in the left-most column of the Young diagram.

The first-column hook lengths uniquely determine a partition $\lambda.$ We can generalize the set of first column hooks using the notion of a {\it bead set} $X$ corresponding to $\lambda$, where $X=\{0,\cdots,k-1,|h_{11}|+k,|h_{21}|+k, |h_{31}|+k,\cdots\}$ for some non-negative
integer $k$. It can also be seen as a finite set of non-negative integers, represented by {\it beads} at integral points of the $x$-axis, i.e., a bead at position $x$ for
each $x$ in $X$ and {\it spacers} at positions not in $X$.  Then
$|X|$ is the number of beads that occur after the zero position, wherever that may fall. The {\it minimal} bead-set $X$ of $\lambda$ is one where 0 labels the first spacer, and is exactly the set of first-column hook lengths.
\begin{example} Suppose $\lambda=(4,3,2).$ Then $\{h_{\iota1}\}=\{2,4,6\}$, where $1\leq\iota\leq 3$ is the set of first column hook lengths, and a minimal bead set.  Note that $X'=\{0,2+1,4+1,6+1\}=\{0,3,5,7\}$ and $X''=\{0,1,2,3,2+4,4+4,6+4\}=\{0,1,2,3,6,8,10\}$ are two bead sets that also correspond to $\lambda.$ 
\end{example}
The set of hooks $\{h_{\iota \gamma}\}$ of $\lambda$ correspond bijectively to pairs $(x,y)$ where $x\in X$, $y\not\in X$ and $x>y$; that is, a bead in a bead-set $X$ of $\lambda$ and a spacer to the left of it. Hooks of length $s$ are those such that $x-y=s$. \\

The following result (Lemma 2.4, \cite{O}) allows us to recover the size of the part from its corresponding bead. 
\begin{lemma} \label{onebead} Let $X$ be a bead-set of a partition $\lambda.$ The size of the part $\lambda_{\alpha}$ of $\lambda$ corresponding to the bead $x'\in X$ is the number of spacers to the left of the bead, that is, $\lambda_{\alpha}=|y\not\in X:y<x'|.$
\end{lemma}
Given a fixed integer $s$, we can arrange the nonnegative integers into an $s${\it -grid}, an array of $s$ columns labeled from $0\leq i\leq s-1$, and consider the columns as runners, on which beads are placed in their respective positions.  This organizes a given bead-set by their values modulo $s$. 

\begin{definition}[$s$-abacus]  Consider a bead-set $X$. Placing a bead in each position on the $s$-grid where there is a value $x\in{X}$ gives the {\bf $\textit{s-abacus diagram}$} ${\mathcal S}$ of $X$. Positions not occupied by beads are {\it spacers}. A {\bf $\textit{minimal}$} $s$-abacus ${\mathcal S}$ corresponds to a minimal bead-set $X$ (where the first spacer labels the zero position). 
\end{definition}
\begin{definition}[$s$-abacus position] \label{index} 
Let ${\mathcal S}$ be the $s$-abacus associated to a bead-set $X$. We say that a bead $x\in X$ has {\bf $\textit{s-abacus position}$} $(i,j)\in {\mathcal S}$, where $0\leq i\leq s-1$ and $j\geq 0$ if and only if $i+js=x\in X$.
\end{definition}
\begin{definition} A {\bf $\textit{sub}$}-{\bf $\textit{abacus}$} ${\mathcal S}'$ of an $s$-abacus ${\mathcal S}$ is a set of $s$-abacus positions $(i,j)$ that obey the property that if $(i,j)\in {\mathcal S}'$, then $(i,j)\in {\mathcal S}.$ 
\end{definition}
\subsection{$s$-core and simultaneous $(s,t)$-core partitions}
A {\it s-core partition} (or simply $s$-core) of $n$ is a partition in which no hook of length $s$ appears in the Young diagram. Note that a bead $x$ in runner $i$ with a spacer $y$ one row below, but also in runner $i$, corresponds to an $s$-hook of $\lambda$. A partition $\lambda$ is a $s$-core if and only if its $s$-abacus has the property that no spacer occurs below a bead in a given runner.  This is expressed in the following lemma.
\begin{lemma}\label{score} An $s$-abacus ${\mathcal S}$ corresponds to an $s$-core partition if and only $(i,j)\in {\mathcal S}$ and $j>0$ implies that $(i,j-1)\in {\mathcal S}.$
\end{lemma}
We then have the following result. 
\begin{corollary}\label{mscore} An $s$-core partition is an $ms$-core partition for all $m>1.$
\end{corollary}
\begin{proof} An $ms$-hook on an $s$-abacus ${\mathcal S}$ is expressed as a bead in abacus position $(i,j)$ and a spacer in position $(i,j-m)$. Either there are no beads in positions $(i,j-1),\cdots,(i,j-m+1)$ or there is at least one.  In the either case, we violate the condition of Lemma \ref{score}. 
\end{proof}
A result of Sylvester from 1884 gives us the size of the largest possible first-column hook length of a simultaneous $(s,t)$-core. 
\begin{proposition} \label{sly} If $gcd(s,t)$=1, the largest possible hook of an $(s,t)$-core has length $st-s-t.$
\end{proposition}

In recent years, the study of core partitions has expanded to include partitions which are simultaneously cores for various integers.   Anderson \cite{A} first enumerated $(s,t)$-cores in the case when $s$ and $t$ are relatively prime. Subsequently, the work of Olsson and Stanton (and others) showed that, when gcd$(s,t)=1$, there is a unique $(s,t)$-core with largest weight, denoted by $\kappa_{s,t}$. We call such a simultaneous core {\it maximal}.   
\begin{theorem}[J. Olsson and D. Stanton, Theorem 4.1, \cite{O-S}] \label{st} Let gcd$(s,t)=1.$ Then there is a unique maximal $(s,t)$-core $\kappa_{s,t}$ such that
\begin{equation*}
|\kappa_{s,t}|=\frac{(s^2-1)(t^2-1)}{24}.
\end{equation*}
\end{theorem}
Using the notation $\kappa_{s,t}$ we restate a canonical result of J. Anderson (Proposition 1,\cite{A}).
\begin{proposition} \label{And} Suppose gcd$(s,t)=1.$ The minimal $s$-abacus ${\mathcal S}$ of an $(s,t)$-core will be a sub-abacus of the minimal $s$-abacus ${\mathcal K}$ of $\kappa_{s,t}$ the maximal $(s,t)$-core partition. Furthermore, if $(i,j)\in {\mathcal S}$ then $(i,j-1)\in {\mathcal S}$ and $(i-t,j)\in {\mathcal S}$ if $i> t$. If $i<t$ then $(i,j)\in {\mathcal S}$ implies $(s-t-1,j-1)\in {\mathcal S}.$
\end{proposition}
As a consequence of Proposition \ref{sly}, Theorem \ref{st}, and Proposition \ref{And}, we have the following useful result.
\begin{corollary} \label{bighook} $\kappa_{s,t}$ is the unique $(s,t)$-core with a hook of length $st-s-t.$
\end{corollary}
We note that the special case of $(s,ms+1)$-cores has attracted particular interest. Early examples of the now-resolved Armstrong conjecture (cf. \cite{A-H-J} \cite{J} \cite{W}) included the $(s,s+1)$ and $(s,ms+1)$ cases, done by F. Zanello and R. Stanley \cite{S-Z} and A. Aggarwal \cite{Ag} respectively. S. Fischel and M. Vazirani \cite{FV}, have studied a bijection between $(s,ms+1)$-cores and dominant Shi regions. 

Self-conjugate simultaneous core partitions are also of interest. B. Ford, H. Mai and L. Sze \cite{F} have, in a manner analogous to Olsson-Stanton, enumerated the self-conjugate $(s,t)$-core partitions. 

In Section 4, we study $(s, ms\pm1)$-cores with distinct parts; in Section 5 we apply our methods to analogize the results of Xiong and Straub and enumerate the self-conjugate simultaneous $(s,s+1)$-core and $(s,ms\pm1)$-core partitions with distinct parts. Before we do, we give an overview of existing results on simultaneous core partitions with distinct parts.
\subsection{Simultaneous $(s,t)$-cores with distinct parts}
Simultaneous core partitions with distinct parts were first introduced as an object of study by T. Amdeberhan. One of the conjectures proposed by T. Amdeberhan (Conjecture 11.9,\cite{Am}) has lead to new results by H. Xiong and A. Straub in this area. More recently A. Zaleski (\cite{Z}) has published some on moments of their generating functions, building on work by S. Ekhad and D. Zeilberger \cite{EZ}.
\begin{theorem}[H. Xiong, Theorem 1.1(1), \cite{X}] \label{Xiong} Let $s\geq 1$ and $F_{s+1}$ be the $(s+1)$st Fibonacci number. Then $F_{s+1}$ is the number of $(s,s+1)$-core partitions with distinct parts.
\end{theorem}
\begin{theorem}[A. Straub, Theorem 4.1, \cite{S}] \label{Straub} Let $m,s\geq 1$. The number $E^-_m(s)$ of $(s,ms-1)$-core partitions with distinct parts is characterized by $E^-_m(1)=1$ and $E^-_m(2)=m$ and, for $s\geq3$,
$$E^-_m(s)=E^-_m(s-1)+mE^-_m(s-2).$$ 
\end{theorem}
Our paper develops a framework from which results of H. Xiong and A. Straub in this direction follow naturally. That is, we use the geometry of the $s$-abacus of the maximal $(s,s+1)$-core, and that of the $ms$-abacus of the maximal $(s,ms-1)$-core and $(s,ms+1)$-core partitions, to prove Theorems \ref{Xiong} and \ref{Straub} in a uniform manner. Before proving Theorem \ref{Straub}, however, we enumerate $(s,ms+1)$-core partitions with distinct parts (Theorem \ref{Straub2}). In doing so, we provide a partition-theoretic meaning to a numerical relation first observed by Straub (see Lemma 4.3, \cite{S}). This lays the groundwork for the proof of Theorem \ref{Straub}.
\begin{theorem} \label{Straub2} Let $m,s\geq 1$. The number $E^+_m(s)$ of $(s,ms+1)$-core partitions into distinct parts is characterized by $E^+_m(1)=1$, $E^+_m(2)=m+1$ and, for $s\geq3,$ 
$$E^+_m(s)=E^+_m(s-1)+mE^+_m(s-2).$$
\end{theorem}
These proofs appear in Section 4.
\subsection{Simultaneous $(s,ms-1,ms+1)$-core partitions}
Suppose $s,t,u$ are positive integers such that gcd$(s,t,u)=1$. Enumerating and calculating the weight of simultaneous $(s,t,u)$-cores is more complicated than simultaneous $(s,t)$-cores, in part because no analogous result to Sylvester's characterization of the maximum possible hook length exists. However, for special cases, progress has been made. T. Amdeberhan and E. Leven \cite{AL}, R. Nath and J. Sellers, \cite{Nath-Sellers}, Xiong \cite{X}, and Yang-Zhang-Zhou \cite{Y-Z-Z} investigated $(s-1,s,s+1)$-cores, and both the weight of the maximal core and the number of such cores is known. V. Wang \cite{W} has enumerated $(s,s+d,s+2d)$-cores. A. Aggarwal \cite{Ag0} has also studied containment properties of $(s,t,u)$-cores. 

The methods described herein also allow us to study another family of triply simultaneous cores: in particular, we calculate the weight of the {\bf longest} $(s,ms-1,ms+1)$-core partition (that is, the core partition with the largest number of parts). 
\begin{theorem} \label{m2} The weight of the longest $(s,ms-1,ms+1)$-core is
\begin{enumerate}
\item $\frac{m^2t(t-1)(t^2-t+1)}{6}$ if $s=2t-1$\\
\item $\frac{m^2(t-1)^2(t^2-2t+3)}{6} - \frac{m(t-1)^2}{2}$ if $s=2t-2.$
\end{enumerate}
\end{theorem}
We also conjecture that this is the weight of any maximal $(s, ms-1, ms+1)$-core.\\

The key observation we utilize in proving all of our results is the way in which the $ms$-abaci of maximal $(s,ms\pm 1)$-cores are built up from the $s$-abaci of $(s,s\pm1)$-cores and other objects. Hence, we now transition to a detailed description of the relevant $s$-abaci and $ms$-abaci.\\

{\bf Note}: For the remainder of the paper, $(s,ms\pm1)$-core partitions will refer to either a $(s, ms-1)$-core partition {\it or} a $(s,ms+1)$-core partition. The notation $(s,ms-1,ms+1)$-core will indicate a core that is simultaneously a $s$-core, an $(ms-1)$-core, {\it and} an $(ms+1)$-core.
\section{$s$-abaci of $(s,s\pm1)$-cores}
The following two lemmas follow from Definition \ref{index}.
\begin{lemma}\label{s+1core} An $s$-abacus ${\mathcal S}$ is an $(s+1)$-core if $(i,j)\in {\mathcal S}$ implies 
\begin{enumerate}
\item $(i-1,j-1)\in {\mathcal S}$ when $0<i\leq s-1$ and $j\geq 1$, and
\item $(s-1,j-2)\in {\mathcal S}$ when $i=0$ and $j\geq 2$.
\end{enumerate}
\end{lemma}
\begin{proof} Suppose ${\mathcal S}$ is the $s$-abacus of an $(s+1)$-core. Then $(i,j)\in {\mathcal S}$ if and only if there if a bead in a position $s+1$ steps to the left, wrapping down-and-around-to-the-right the abacus when necessary. This is exactly the statement of the lemma.
\end{proof}
\begin{lemma} \label{notcore} An $s$-abacus ${\mathcal S}$ represents an $(s-1)$-core partition if $(s-1,0)\not\in {\mathcal S}$ and $(i,j)\in {\mathcal S}$ with $j>0$ implies 
\begin{enumerate}
\item $(i+1,j-1)\in {\mathcal S}$ when $0< i< s-1$ 
\item $(0,j)\in {\mathcal S}$ when $(s-1,j)$ is.
\end{enumerate}
\end{lemma}
\begin{proof} The argument is identical to that of Lemma \ref{s+1core} replacing $s+1$ by $s-1$, with the caveat that a bead in position $(s-1,0)$ is not permitted.
\end{proof}
A crucial part of our argument below will involve the following two abaci constructions.
\begin{definition}\label{AbS} Let ${\mathcal A}(s)$ be the $s$-abacus with beads in abacus positions $(i,j)$ for every $(i,j)$ such that $0<i\leq s-1$ and $0\leq j \leq i-1$.
\end{definition}
\begin{example}  ${\mathcal A}(5)=\{(1,0),(2,0),(2,1),(3,0),(3,1),(3,2),(4,0),(4,1),(4,2),(4,3)\}.$ [See Figure 1.]
\end{example}

{\scriptsize
\begin{figure}[h!]
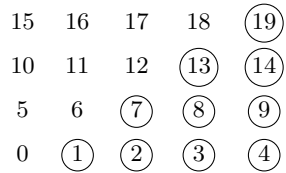

\label{abacus789}
\begin{center}
\[
\begin{array}{ccccc}
15 & 16 & 17 & 18 & \encircle{19}\\
10 & 11 & 12 & \encircle{13} & \encircle{14}\\
{5} & 6 & \encircle{7} & \encircle{8} & \encircle{9}\\
0 & \encircle{1} & \encircle{2} & \encircle{3}& \encircle{4}\\
\end{array} 
\] 

\caption{${\mathcal A(5)}$}

\end{center}
\end{figure}
}

\begin{lemma} \label{topS} ${\mathcal A}(s)$ is the minimal $s$-abacus of $\kappa_{s,s+1}$, the maximal $(s,s+1)$-core.
\end{lemma}
\begin{proof} 
${\mathcal A}(s)$ is minimal by construction. To show ${\mathcal A}(s)$ is the $s$-abacus of an $(s,s+1)$-core we have to show that Lemma \ref{score} and Lemma \ref{s+1core} are satisfied.
Suppose $j>0$. If $(i,j)\in {\mathcal A}(s)$ then, since $j\leq i-1$, it follows by Definition \ref{AbS} that when $1\leq i\leq s-1$, we have  $(i,j-1)\in {\mathcal A}(s)$ and $(i-1,j-1)\in {\mathcal A}(s)$.

Let $X$ be the underlying bead-set of ${\mathcal A}(s)$. Since $(s-1,s-2)\in {\mathcal A}(s)$, by Definition \ref{index} we have $s-1+(s-2)s=s^2+s-1\in X$. By Corollary \ref{bighook} this implies that ${\mathcal A}(s)$ is the $s$-abacus of the maximal $(s,s+1)$-core, since $s(s+1)-s-(s+1)=s^2-s-1.$
\end{proof}
\begin{definition} \label{bee} Let ${\mathcal B}_k(s)$ be the $s$-abacus with beads in abacus positions $(i,j)$ for every $(i,j)$ such that $0<i\leq s-1-k$ and $0\leq j\leq s-i-1.$
\end{definition}
The proofs of Lemmas \ref{yikes} and \ref{tippy} follow from Definition \ref{bee}. Details are left to the reader.
\begin{lemma} \label{yikes} ${\mathcal B}_k(s)$ has the following properties. If $(i,j)\in {\mathcal B}_k(s)$ then
\begin{enumerate}
\item $(i,j-1)\in {\mathcal B}_k(s)$ and
\item $(i+1,j-1)\in {\mathcal B}_k(s)$. 
\end{enumerate}
\end{lemma}
\begin{lemma} \label{tippy} ${\mathcal B}_{1}(s)$ is obtained from ${\mathcal B}_{0}(s)$ by removing beads in abacus-positions $(i,s-1-i)$ as $1\leq i\leq s-1.$
\end{lemma}
\begin{example} \label{b05} ${\mathcal B}_0(5)=\{(1,0),(1,1),(1,2),(1,3),(2,0),(2,1),(2,2),(3,0),(3,1),(4,0)\}$. [See Figure 2.]
\end{example}

{\scriptsize
\begin{figure}[h!]
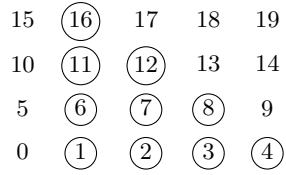

\label{abacus91011}
\begin{center}

\[
\begin{array}{ccccc}
15 & \encircle{16} & 17 & 18 & 19\\
10& \encircle{11} & \encircle{12}& 13 & 14\\
5& \encircle{6} & \encircle{7} & \encircle{8} & 9 \\
0 & \encircle{1} & \encircle{2} & \encircle{3} & \encircle{4}\\
\end{array} 
\] 

\caption{${\mathcal B_0(5)}$: $5$-abacus of a $(5,9)$-core}

\end{center}
\end{figure}
} 
\begin{example} \label{b15} ${\mathcal B}_1(5)=\{(1,0),(1,1),(1,2),(2,0),(2,1),(3,0)\}$. [See Figure 3.]
\end{example}
{\scriptsize
\begin{figure}[h!]
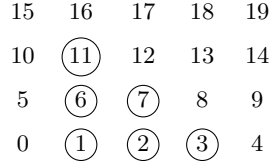

\label{abacus91011}
\begin{center}

\[
\begin{array}{ccccc}
15 & 16 & 17 & 18 & 19\\
10& \encircle{11} & 12 & 13 & 14\\
5& \encircle{6} & \encircle{7} & 8 & 9 \\
0 & \encircle{1} & \encircle{2} & \encircle{3} & 4\\
\end{array} 
\] 

\caption{${\mathcal B_1(5)}$: $5$-abacus of the maximal $(4,5)$-core}

\end{center}
\end{figure}
} 
For the remainder of the paper, we focus our attention on ${\mathcal B}_0(s)$ and ${\mathcal B}_1(s)$.
\begin{lemma} \label{topB} $\mathcal{B}_1(s)$ is the minimal $s$-abacus of the maximal $(s-1,s)$-core partition.
\end{lemma}
\begin{proof}  ${\mathcal B}_1(s)$ is minimal by construction. By Lemma \ref{score}, Lemma \ref{notcore} and Lemma \ref{yikes}, ${\mathcal B}_1(k)$ is a $(s-1,s)$-core. 

Let $X$ be the underlying bead-set of ${\mathcal B}_1(s)$. Since $(1,s-3)\in {\mathcal B}_1(s)$, $1+(s-3)s=(s-1)s-s-(s-1)\in X.$ By Corollary \ref{bighook}, we are done.
\end{proof}
\begin{corollary} \label{subb} ${\mathcal B}_{1}(s)$ is a sub-abacus of ${\mathcal B}_{0}(s).$
\end{corollary}
\begin{proof} Follows from Lemma \ref{tippy}.
\end{proof}
\begin{example} ${\mathcal B}_1(5)$ is a sub-abacus of ${\mathcal B}_0(5).$ [See Figures 2 and 3.]
\end{example}
\section{$ms$-abaci of $(s,ms\pm1)$-cores}
We now generalize the results of the previous section by moving to $(s,ms\pm1)$-cores.
\begin{lemma}\label{smscore} Let ${\mathcal M}$ be an $ms$-abacus, where $m>1$. Then ${\mathcal M}$ corresponds to an $s$-core partition if 
\begin{enumerate}
\item $(i,j)\in {\mathcal M}$ then $(i-s,j)$ when $s\leq i\leq ms-1$
\item $(i,j)\in {\mathcal M}$ then $(s-i-1,j-1)\in {\mathcal M}$ if $0\leq i<s.$
\end{enumerate} 
\end{lemma}
\begin{proof} Part (1) is immediate. Part (2) ensures that when moving $s$ positions to the left of $(i,j)$ wraps around-and-down the $ms$-abacus, a bead occupies the relevant abacus position.
\end{proof}
The next corollary follows from Corollary $\ref{mscore}$ and the definition of an $ms$-abacus.
\begin{corollary} \label{smcore} Let ${\mathcal M}$ be an $ms$-abacus. If ${\mathcal M}$ is an $s$-core, then, if $(i,j)\in {\mathcal M}$, we have $(i,j-1)\in {\mathcal M}.$
\end{corollary}
\begin{definition}  We define the following two special $ms$-abaci.
\begin{enumerate}
\item Let ${\mathcal E_m^-}(s)$ be the $ms$-abacus with beads in abacus positions $(i+\ell s,j),$ where $0\leq \ell\leq m-2$ for $1\leq i\leq s-1$ and $1\leq j\leq s-i-1$ and $(i+(s-1)m,s-i-2)$ for $1\leq i\leq s-2$ and $0\leq j\leq s-i-2$.
\item Let ${\mathcal E_m^+}(s)$ be the $ms$-abacus defined by $(i+\ell s,j)$ where $1\leq i\leq s-1$ and $0\leq j\leq i-1$ and $0\leq \ell \leq m-1$.
\end{enumerate}
\end{definition}
\begin{example} Let ${\mathcal E}^-(5,\ell)=\cup_{0\leq i\leq 4}\cup_{0\leq j\leq 4-i}(i+5\ell,j)$. Then
$${\mathcal E}^-_3(5)=\cup_{0\leq \ell \leq 2}{\mathcal E^-(5,\ell)}.$$ [See Figure 4.]
\end{example}
{\scriptsize
\begin{figure}[h!] \label{e5minus}
\label{abacus789}
\begin{center}

\[
\begin{array}{ccccccccccccccc}
{45} & \encircle{46} & 47 & 48 & 49 & 50 & \encircle{51} & 52 & 53 & 54 & 55 & 56 & 57 & 58 & 59\\
{30} & \encircle{31} & \encircle{32} & 33 & 34 & 35 & \encircle{36} & \encircle{37} & 38 & 39 & 40 & \encircle{41} & 42 & 43 & 44\\
15 & \encircle{16} & \encircle{17} & \encircle{18} & 19 & 20 & \encircle{21} & \encircle{22} & \encircle{23} & 24 & 25 & \encircle{26} & \encircle{27} & 28 & 29\\
0 & \encircle{1} & \encircle{2} & \encircle{3} & \encircle{4} & 5 & \encircle{6} & \encircle{7} & \encircle{8} & \encircle{9} & 10 & \encircle{11} & \encircle{12} & \encircle{13} & 14\\
\end{array} 
\] 

\caption{${\mathcal E_3^-}(5)$: $15$-abacus of the maximal $(5,14)$-core}

\end{center}
\end{figure}
} 
\begin{example} Let ${\mathcal E}^+(5,\ell)=\cup_{0\leq i\leq 4}\cup_{0\leq j\leq i-1}(i+5\ell,j)$. Then
$${\mathcal E}^+_3(5)=\cup_{0\leq \ell \leq 2}{\mathcal E^+(5,\ell)}.$$ [See Figure 5.]
\end{example}
{\scriptsize
\begin{figure}[h!]
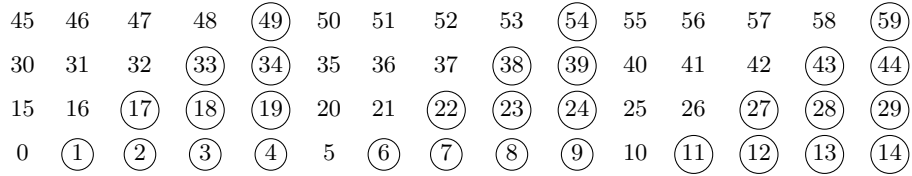
 \label{e5plus}
\label{abacus789}
\begin{center}

\[
\begin{array}{ccccccccccccccc}
45 & 46 & 47 & 48 & \encircle{49} & 50 & 51 & 52 & 53 & \encircle{54} & 55 & 56 & 57 & 58 & \encircle{59}\\
30 & 31 & 32 & \encircle{33} & \encircle{34} & 35 & 36 & 37 & \encircle{38} & \encircle{39} & 40 & 41 & 42 & \encircle{43} & \encircle{44}\\
15 & 16 & \encircle{17} & \encircle{18} & \encircle{19} & 20 & 21 & \encircle{22} & \encircle{23} & \encircle{24} & 25 & 26 & \encircle{27} & \encircle{28} & \encircle{29}\\
0 & \encircle{1} & \encircle{2} & \encircle{3} & \encircle{4} & 5 & \encircle{6} & \encircle{7} & \encircle{8} & \encircle{9} & 10 & \encircle{11} & \encircle{12} & \encircle{13} & \encircle{14}\\
\end{array} 
\] 

\caption{${\mathcal E_3^+}(5)$: $15$-abacus of the maximal $(5, 16)$-core}

\end{center}
\end{figure}
} 
\begin{theorem} \label{Emax} The following relations hold for ${\mathcal E_m^-}(s)$ and ${\mathcal E_m^+}(s)$.
\begin{enumerate}
\item ${\mathcal E^-_m}(s)$ is the minimal $ms$-abacus of the maximal $(s, ms-1)$-core partition.
\item ${\mathcal E_m^+}(s)$ is the minimal $ms$-abacus of the maximal $(s,ms+1)$-core partition.
\end{enumerate}
\end{theorem}
Before we can prove Theorem \ref{Emax}, we need to define a new operation on abaci.
\begin{definition} \label{Oplus} Suppose ${\mathcal A}$ and ${\mathcal B}$ are $s$-abaci and $t$-abaci respectively.  We denote by ${\mathcal A}\wedge {\mathcal B}$ the $(s+t)$-abacus whose $0\leq i\leq s-1$ runners correspond to the $s-1$ runners of ${\mathcal A}$, and whose $s\leq i\leq s+t-1$ runners of correspond to the $t-1$ runners of ${\mathcal B}$. This will be called {\bf appending} $\mathcal{B}$ to $\mathcal{A}$ on the right.
When we append ${\mathcal A}$ to itself $m$ times, we will use the notation $\wedge_m{\mathcal A}=\underbrace{{\mathcal A}\wedge\cdots\wedge{\mathcal A}}_{\text{m}}$.
\end{definition}
\begin{lemma} \label{BAM} The following relations hold.
\begin{enumerate}
\item ${\mathcal E_m^-}(s)=(\wedge_{m-1}{\mathcal B}_0(s))\wedge{\mathcal B}_1(s)$\\
\item ${\mathcal E_m^+}(s)=\wedge_m{\mathcal A}(s)$
\end{enumerate}
\end{lemma}
\begin{proof}  Fix an $\ell\in[0,m-1]$. Consider the projection map $\pi_{\ell}$ that takes $(i+\ell s,j)$ to $(i,j)$, where $1\leq i\leq s-1$. We prove each case separately.
\begin{enumerate} 
\item Fix a $\ell\in [0,m-2]$. Then, under $\pi_{\ell}$, the beads with abacus positions $(i+\ell s,j)$ $1\leq i\leq s-1$, and $1\leq j\leq s-i-1$ are in bijection with those in ${\mathcal B}_0(s)$. For $\ell=m-1$,  under $\pi_{m-1}$, beads in positions $(i+(s-1)m,s-i-2)$ where $1\leq i\leq s-2$ are in bijection with ${\mathcal B}_1(s)$.\\
\item Fix an $\ell\in [0,m-1]$. Then, under $\pi_{\ell}$, the beads with abacus positions $(i+\ell s,j)$ $1\leq i\leq s-1$, and $1\leq j\leq s-i-1$ are in bijection with those in ${\mathcal A}(s)$.  
\end{enumerate}
\end{proof}
We are now in a position to prove Theorem \ref{Emax}, which is critical for the rest of our results. 
\begin{proof}[Proof of Theorem \ref{Emax}]
The abaci above are minimal, by construction. It remains to show they satisfy the relevant core properties and that they are of maximal weight. We consider each case separately.
\begin{enumerate}
\item Suppose $(i,j)\in {\mathcal E}_m^-(s)$, with $j>0$. By Lemma \ref{BAM}, $(i,j)\in \wedge_{m-1}({\mathcal B}_0(s))\wedge{\mathcal B}_1(s)$. To see that ${\mathcal E}_m^-(s)$ is an $s$-core, we have to satisfy the conditions of Lemma \ref{smscore}. Suppose $j>0$. If $(i,j)$ is in the copy of ${\mathcal B}_1(s)$, then $(i-s,j)\in {\mathcal E}^-_m(s)$, since ${\mathcal B}_1(s)$ is a sub-abacus of ${\mathcal B}_0(s)$ by Lemma \ref{subb}. If $(i,j)$ is in one of the rightmost $m-2$ copies of ${\mathcal B}_1(s)$, then, $(i-s,j)\in {\mathcal E}^-_m(s).$ If $(i,j)$ is in the leftmost copy ${\mathcal B}_1(s)$, notice that $\pi_0(i)=i.$  It is enough to see that there exists an $(i+(m-1)s,j-1)$ in the rightmost copy of ${\mathcal B}_1(s)$, using $\pi_{m-1}$ and Lemma \ref{tippy}. 

To show that $\mathcal{E}^-_m(s)$ is an $(ms-1)$-core, we map an abacus position $(i+ks,j)$, where $j>0$, to its local coordinates $(i,j)$ via $\pi_k$. By Lemma \ref{notcore}, we know that $(i+1,j-1)$ is in ${\mathcal B}_{\ell}(s)$ where $\ell$ is either 0 or 1. Mapping this local abacus position back to the $ms$-abacus we conclude $(i+1+ks,j-1)\in {\mathcal E}^-_m(s).$ Since $(ms-1,0)\not\in {\mathcal E}^-_m(s)$ by construction, ${\mathcal E}^-_m(s)$ satisfies the criteria for an $(ms-1)$-core.

Finally we consider position $((m-2)s+1,s-2)\in {\mathcal E}^-_m(s)$. By Definition \ref{index} this corresponds to the bead-value $s(ms-1)-s-ms+1$, which by Corollary \ref{bighook} means this partition is the $(s,ms-1)$-core of maximal weight.\\ 
\item The proof is analogous to (1); the abacus position $(ms-1,s-2)$ corresponds to the maximal bead value in the underlying bead-set.
\end{enumerate}
\end{proof}
\section{$ms$-abaci of $(s,ms\pm1)$-cores with distinct parts}
We now wish to turn our attention to simultaneous cores with distinct parts.  This will allow us to provide unified proofs of Theorems \ref{Xiong}, \ref{Straub} and \ref{Straub2}. 
\begin{lemma} \label{stinct} The partition $\lambda$ with minimal $s$-abaci ${\mathcal S}$ has distinct parts if and only if $(i,j)\in {\mathcal S}$ implies that 
\begin{enumerate}
\item $(i-1,j)\not\in {\mathcal S}$ and $(i+1,j)\not\in {\mathcal S}$ if $1<i<s-1,$ and 
\item $(i-1,j)\not\in {\mathcal S}$ if $i=s-1.$ 
\end{enumerate}
\end{lemma}
\begin{proof} A partition has distinct parts if and only if its minimal bead-set $X$ satisfies the following property: if $x,y\in X$ and $x>y$, then $x-y\neq 1.$ This is exactly the statement of the lemma when translated into $s$-abaci. 
\end{proof}
The combination of Lemma \ref{stinct} and Theorem \ref{Emax} allow us to study the abaci of certain simultaneous core partitions with distinct parts. 
\begin{lemma} \label{row} Let ${\mathcal A}(s)$, ${\mathcal E}^-_m(s)$, and ${\mathcal E}^+_m(s)$ be as above.
\begin{enumerate}
\item The minimal $s$-abacus ${\mathcal S}$ of any $(s,s+1)$-core with distinct parts will be a sub-abacus of ${\mathcal A}(s)$ consisting of beads taken only from its first row. 
\item The minimal $ms$-abacus ${\mathcal M}^-$ of any $(s,ms-1)$-core partition with distinct parts will be a sub-abacus of ${\mathcal E}^-_m(s)$ consisting of beads taken only from its first row. 
\item The minimal $ms$-abacus ${\mathcal M}^+$ of any $(s,ms+1)$-core partition will be a sub-abacus of ${\mathcal E}^+_m(s)$ consisting only of beads taken only from its first row. 
\end{enumerate}
\end{lemma}
\begin{proof} We know by Proposition \ref{And} that ${\mathcal S}$, ${\mathcal M}^-$ and ${\mathcal M}^+$ are sub-abaci of ${\mathcal A}(s)$, ${\mathcal E}_m^+(s)$ and ${\mathcal E}_m^-(s)$. 
\begin{enumerate}
\item Suppose $(i,j)\in {\mathcal S}$ such that $j>0.$ Then $(i,j-1)\in {\mathcal S}$ and $(i-1,j-1)\in {\mathcal S}$ by Proposition \ref{And}, Lemma \ref{s+1core}, Lemma \ref{notcore}, Lemma \ref{topS}. This is a contradiction. 
\item Suppose $(i,j)\in {\mathcal M}^-$ such that $j>0.$ Then  $(i+1,j-1)\in {\mathcal M}^-$ by Lemma \ref{notcore} and Theorem \ref{Emax}(1). However, $(i,j-1)\in {\mathcal M}^-$ by Proposition \ref{And}, Corollary \ref{mscore}, Lemma \ref{smscore}. This is a contradiction. 
\item Suppose $(i,j)\in {\mathcal M}^+$ such that $j>0.$ Then $(i-1,j-1)\in {\mathcal M}^+$ by Lemma \ref{s+1core} and Theorem \ref{Emax}(2). However $(i,j-1)\in {\mathcal M}^+$ for the same reason as in (2).
\end{enumerate}
\end{proof}
We now possess all of the necessary tools to prove Theorems \ref{Xiong}--\ref{Straub2} in a unified, combinatorial fashion.  As noted earlier, we switch our convention and prove Theorem \ref{Straub2} first; Theorem \ref{Straub} then follows from Theorem \ref{middle} and some manipulation.
\begin{proof}[Proof of Theorem \ref{Xiong}] 
There is only one simultaneous $(1,2)$-core partition with distinct parts, the empty partition. There are two simultaneous $(2,3)$-core partitions with distinct parts: the empty partition, and $\lambda=(1)$. This gives us the initial conditions, $F_2=1$ and $F_3=2$.

By Lemma \ref{row}(1), for any $s$-abacus ${\mathcal S}$ of an $(s,s+1)$-core with distinct parts, if $(i,j)\in {\mathcal S}$ then $j=0$ where $0\leq i\leq s-1$. We divide the count into two cases, depending on whether or not $(s-1,0)\in {\mathcal S}$. 

If $(s-1,0)\in {\mathcal S}$, then by Lemma \ref{stinct}, $(s-2,0)\not\in {\mathcal S}$. By considering only the runners $0\leq i\leq s-3$, we conclude there are $F_{s-1}$ possible $s$-abacus arrangements for ${\mathcal S}$ with $(s-1,0)\in {\mathcal S}$. If $(s-1,0)\not\in {\mathcal S}$, then by considering only the runners $0\leq i\leq s-2,$ we can conclude that there are $F_{s}$ possible $s$-abacus arrangements for ${\mathcal S}$ with $(s-1,0)\not\in {\mathcal S}$. Hence the total number of acceptable $s$-abacus arrangements for an $(s,s+1)$-core with distinct parts is $F_{s+1}=F_s+F_{s-1}.$ This completes the proof.\\
\end{proof}
\begin{proof}[Proof of Theorem \ref{Straub2}] There is only one simultaneous $(1,m+1)$-core; the empty partition. By Lemma \ref{BAM} and Lemma \ref{row}(3), there are $m$ simultaneous $(2,2m+1)$-core partitions with distinct parts; the empty partition, and one partition for each set of abacus positions $\{\cup^{m'}_{\ell=0}(1+2\ell ,0)\}$ where $m'\in[0,m-1].$

By Lemma \ref{row}(3) for any $s$-abacus ${\mathcal M}^+$ of a $(s,ms-1)$-core with distinct parts, if $(i,j)\in {\mathcal M}^+$, then $(i,j)=(i+\ell s,0)$ where $0< i\leq s-1$ when $0\leq \ell\leq m-1$. We divide the count into two cases: where $(s-1,0)$ is in ${\mathcal M}^+$, or where it is not.

Suppose first that $(s-1,0)\in {\mathcal M}^+$. Then $(ks-2,0)\not\in {\mathcal M}^-$ for $0\leq k \leq m-1$, by Lemma \ref{stinct}. So we can consider only the $i+\ell s$ where $0<i\leq s-3$ for $0\leq \ell \leq m-1$: the number of such acceptable $m(s-2)$-abacus arrangements is $E^+_m(s-2)$. However there are $m-1$ additional positions $(2s-1,0),(3s-1,0),(4s-1,0),\cdots,(ms-1,0)$ that can also be included without violating Lemma \ref{stinct}. So the total number of acceptable $ms$-abaci from this case is $mE^+_m(s-2).$ 

Suppose $(s-1,0)\not\in {\mathcal M}^+_m(s)$ then $(s-1+\ell s,0)\not\in {\mathcal M}^+$ for $0\leq \ell \leq m-2$. We consider only the $i+\ell s$ where $0<i\leq s-2$ and $0\leq \ell \leq m-1$: there are $E^+_m(s-1)$ possible abacus arrangements. The result follows.\\
\end{proof}
\begin{theorem} \label{middle} $E^-_m(s)=E^+_{m}(s-1)+(m-1)E^+_{m}(s-2).$
\end{theorem}
\begin{proof} There is only one simultaneous $(1,m-1)$-core, the empty partition. By Lemma \ref{BAM} and Lemma \ref{row}(2), there are $m-1$ simultaneous $(2,2m-1)$-core partitions with distinct parts; namely the empty partition, plus one partition for each set of abacus positions $\cup^{m'}_{\ell=0}(1+2\ell,0)$ where $m'\in[0,m-2].$

By Lemma \ref{row}(2) for any $s$-abacus ${\mathcal M}^-$ of a $(s,ms-1)$-core with distinct parts, if $(i,j)\in {\mathcal M}^-$, then $(i,j)=(i+\ell s,0)$ where $0< i\leq s-1$ when $0\leq \ell\leq m-2,$ and $0< i\leq s-2$ when $\ell=m-1.$ We divide our count into two cases, depending on whether or not $(s-1,0)\in {\mathcal M}^-$.

Suppose first that $(s-1,0)\in {\mathcal M}^-$. Then $(ks-2,0)\not\in {\mathcal M}^-$ for $0\leq k \leq m-1$, by Lemma \ref{stinct}. So we can consider only the $i+\ell s$ runners, where $0<i \leq s-3$ and $0\leq \ell \leq m-1$: the number of possible abacus arrangements is $E^-_m(s-2)$.  However there are $m-2$ additional positions $(2s-1,0),(3s-1,0),(4s-1,0),\cdots,(m-1)s-1,0)$ that can also be included without violating Lemma \ref{stinct}. So the total number of acceptable $ms$-abaci from this case is $(m-1)E^+_m(s-2).$ 

Suppose $(s-1,0)\not\in {\mathcal M}^-_m(s)$. Then we can consider only the $i+\ell s$ runners, where $0\leq i\leq s-2$ for $0\leq \ell \leq m-1$: the number of such acceptable $m(s-2)$-abaci arrangements is $E^+_m(s-1)$. Then the total number of acceptable $ms$-abaci arrangements is $E^+_m(s-1)+(m-1)E^+_m(s-2)$.
\end{proof}
Theorem \ref{Straub} follows using a purely algebraic manipulation first employed by Straub. 
\begin{proof}[Proof of Theorem \ref{Straub}]
By Theorem \ref{middle}, we know $E^-_m(s)=E^+_m(s-1)+(m-1)E_m^+(s-2)$. By Theorem \ref{Straub2} we have $E^+_m(s-1)=E^+_m(s-1)+mE_m^+(s-2)$ and $E^+_m(s-2)=E^+_m(s-3)+mE_m^+(s-4)$. Substituting, we get $$E^-_m(s)=E^+_m(s-1)+mE_m^+(s-2)+(m-1)(E^+_m(s-3)+mE_m^+(s-4)).$$ Expanding, we have
$$E^+_m(s-1)+mE_m^+(s-2)+(m-1)E^+_m(s-3)+(m-1)mE_m^+(s-4).$$
Rearranging terms, we arrive at $E^-_m(s)=E^-_m(s-1)+mE^-_m(s-2)$.
\end{proof}
\section{the $ms$-abacus of the longest $(s,ms-1,ms+1)$-core}
We now move to discuss triply simultaneous core partitions. Lemmas \ref{intzero}, \ref{intone}, \ref{pyra} and Corollary \ref{pbase} follow from the relevant definitions. In the interest of brevity the proofs are omitted.
\begin{definition}\label{inter} Let $A$ and $B$ each be $s$-abaci. Then the intersection of $A$ and $B$, denoted $A\cap B$, is the sub-abacus of all beads in both $A$ and $B$. 
\end{definition}
\begin{definition}\label{pyra} Let $\mathcal{S}$ be an $s$-abacus. We say $\mathcal{S}$ is an $s-${\bf$\text{pyramid}$} with base ${\bf [\gamma,\gamma']}$ if when the first row consists of abacus positions $(i,0)$, where $\gamma\leq i\leq \gamma'$, then the second row consists of positions $(i,1)$ where $\gamma'+1\leq i\leq \gamma-1$, and the third row consists of beads in abacus position $(i,2)$ where $\gamma+2\leq i\leq \gamma'-2$, and so on.
\end{definition}
We let ${\mathcal C}_k(s)={\mathcal A}(s)\cap {\mathcal B}_k(s)$.
\begin{lemma}\label{intzero} Let ${\mathcal C}_0(s)=A(s)\cap B_0(s).$ Then ${\mathcal C}_0(s)$ contains beads at all positions $(i,j)$ where $(i,j)$ is such that
\begin{enumerate}
\item $0\leq j\leq i-1$ if $0< i\leq \left \lfloor{\frac{s-1}{2}}\right \rfloor$
\item $0\leq j\leq s-i-1$ if $\left \lfloor{\frac{s+1}{2}}\right \rfloor\leq i\leq s-1.$
\end{enumerate}
\end{lemma}
\begin{proof} Follows by construction.
\end{proof}
\begin{example} ${\mathcal C}_0(5)=\{(1,0),(2,0),(3,0),(4,0),(2,1),(3,1)\}$. [See Figure 6.]
\end{example}
{\scriptsize
\begin{figure}[h!] \label{Czero5}
\label{abacus91011}
\begin{center}

\[
\begin{array}{ccccc}
15 & 16 & 17 & 18 & 19\\
10& 11 & 12& 13 & 14\\
5& 6 & \encircle{7} & \encircle{8} & 9 \\
0 & \encircle{1} & \encircle{2} & \encircle{3} & \encircle{4}\\
\end{array} 
\] 

\caption{${\mathcal C}_0(5)={\mathcal A(5)}\cap{\mathcal B_0(5)}$}

\end{center}
\end{figure}
} 
\begin{lemma}\label{intone} Let ${\mathcal C}_1(s)={\mathcal A}(s)\cap {\mathcal B}_1(s)$. Then ${\mathcal C}_1(s)$ contains beads at all positions $(i,j)$ where $(i,j)$ is such that
\begin{enumerate}
\item $0\leq j\leq i-1$ if $0< i\leq \left \lfloor{\frac{s-1}{2}}\right \rfloor$ and
\item $0\leq j\leq s-i-2$ if $\left \lfloor{\frac{s+1}{2}}\right \rfloor\leq i < s-1.$
\end{enumerate}
\end{lemma}
\begin{example} ${\mathcal C}_1(5)=\{(1,0),(2,0),(3,0),(2,1)\}$. [See Figure 7.]
\end{example}
{\scriptsize
\begin{figure}[h!]
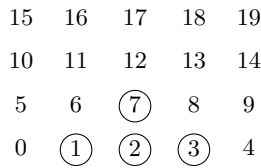
 \label{Cone5}
\label{abacus91011}
\begin{center}

\[
\begin{array}{ccccc}
15 & 16 & 17 & 18 & 19\\
10& 11 & 12 & 13 & 14\\
5& 6 & \encircle{7} & 8 & 9 \\
0 & \encircle{1} & \encircle{2} & \encircle{3} & 4\\
\end{array} 
\] 

\caption{${\mathcal C}_1(5)={\mathcal A(5)}\cap{\mathcal B_1(5)}$}

\end{center}
\end{figure}
} 
\begin{corollary} \label{pbase} Let $C_0(s)$ and $C_{1}(s)$ as above. Then
\begin{enumerate}
\item $C_0(s)$ is a pyramid with base $[1,s-1]$.
\item $C_1(s)$ is a pyramid with base $[1,s-2]$.
\end{enumerate}
\end{corollary}
\begin{lemma} \label{pyra} Let ${\mathcal S}$ be an $s$-abacus. If ${\mathcal S}$ is a pyramid, and $(i,j)\in {\mathcal S}$, where $j>0$ then the following holds:
\begin{enumerate}
\item $(i+1,j-1)\in {\mathcal S}$
\item $(i-1,j-1)\in {\mathcal S}$.
\end{enumerate}
\end{lemma}
\begin{example} $C_0(5)$ is a pyramid with base $[1,4]$. $C_1(5)$ is a pyramid with base $[1,3]$. They both satisfy Lemma \ref{pyra}. [See Figures 6 and 7.]
\end{example}
\begin{lemma} \label{wedgecap} Suppose $A$ and $B$ are $s$-abaci and $A'$ and $B'$ are $t$-abaci. Then
$$(A\wedge A')\cap (B\wedge B')=(A\cap B)\wedge (A'\cap B').$$
\end{lemma}
\begin{proof} This follows from Definitions \ref{Oplus} and \ref{inter}.
\end{proof}
\begin{lemma} ${\mathcal E}^-_m(s)\cap{\mathcal E}^+_m(s)=(\wedge_{m-1}{\mathcal C}_0(s))\wedge {\mathcal C}_1(s).$
\end{lemma}
\begin{proof} By a repeated use of Lemma \ref{wedgecap}, it is enough to look at the intersection of each of the constituent $s$-abaci of ${\mathcal E}^{\pm}_m(s)$ at each of the wedge positions $\ell$, as $0\leq \ell\leq m.$ The result follows from Lemmas \ref{intzero} and \ref{intone}.
\end{proof}
\begin{lemma} \label{csub} ${\mathcal C}_1(s)$ is a sub-abacus of ${\mathcal C}_0(s)$.
\end{lemma}
\begin{proof} This follows from Corollary \ref{subb} and the definitions of ${\mathcal C}_1(s)$ and ${\mathcal C}_0(s)$.
\end{proof}
The simultaneous $(a,b,c)$-core partition with the most parts is called the {\bf longest} one.
\begin{lemma} Let ${\mathcal L}_m(s)$ is the minimal $ms$-abacus of the longest $(s,ms-1,ms+1)$-core.  Then
$${\mathcal L}_m(s)=(\wedge_{m-1}{\mathcal C}_0(s))\wedge {\mathcal C}_1(s).$$
\end{lemma}
\begin{proof} It is enough to show that ${\mathcal L}_m(s)$ is an $(s,ms-1,ms+1)$-core, and that the inclusion of beads in any other abacus positions in ${\mathcal E}^-_m(s)$ or ${\mathcal E}^+_m(s)$ will violate the $(ms-1)$ or $(ms+1)$-core condition. To see it is an $s$-core, we consider a bead in three abacus positions: in the rightmost ${\mathcal C}_1(s)$, the leftmost ${\mathcal C}_0(s)$, or one of the $m-2$ wedge-copies of ${\mathcal C}_0(s)$ in the middle. For a bead in the rightmost ${\mathcal C}_1(s)$; by Lemma \ref{csub}, there is a bead $s$-positions to the left and in the same row, since ${\mathcal C}_1(s)$ is a sub-abacus of ${\mathcal C}_0(s)$.  The same argument applies to beads in the middle $m-2$ copies of ${\mathcal C}_0(s).$  Suppose a bead is in the leftmost copy of ${\mathcal C}_0(s)$ with abacus position $(i,j)$, where $j>0$. Then it is enough that $((ms-1)-i-1,j-1)\in {\mathcal L}_m(s).$  This follows by the construction of ${\mathcal C}_1(s)$, and the projection map $\pi_{m-1}$.

To see that ${\mathcal L}_m(s)$ is an $(ms-1,ms+1)$-core, it is enough to use the projection maps $\pi_{\ell}$ for $0\leq \ell\leq m-1$, and the Lemma \ref{pyra}. Finally, to see that ${\mathcal L}_m(s)$ is longest such, consider the inclusion of a bead in an $ms$-abacus position in ${\mathcal E}^+_m(s)$ or ${\mathcal E}^-_m(s)$ but outside of ${\mathcal E}^-_m(s)\cap{\mathcal E}^+_m(s)$. In this case, either an $(ms-1)$-hook or an $(ms+1)$-hook will arise from a spacer in either the position down-and-to-the-right, or down-and-to-the-left.
\end{proof}
{\scriptsize
\begin{figure}[h!]
\label{abacus789}
\begin{center}

\[
\begin{array}{ccccccccccccccc}
45 & 46 & 47 & 48 & 49 & 50 & 51 & 52 & 53 & 54 & 55 & 56 & 57 & 58 & 59\\
30 & 31 & 32 & 33 & 34 & 35 & 36 & 37 & 38 & 39 & 40 & 41 & 42 & 43 & 44\\
15 & 16 & \encircle{17} & \encircle{18} & 19 & 20 & 21 & \encircle{22} & \encircle{23} & 24 & 25 & 26 & \encircle{27} & 28 & 29\\
0 & \encircle{1} & \encircle{2} & \encircle{3} & \encircle{4} & 5 & \encircle{6} & \encircle{7} & \encircle{8} & \encircle{9} & 10 & \encircle{11} & \encircle{12} & \encircle{13} & 14\\
\end{array} 
\] 

\caption{${\mathcal L}_3(5)={\mathcal E}_3^-(5)\cap{\mathcal E}_3^+(5)$}

\end{center}
\end{figure}
} 
\begin{proof}[Proof of Theorem \ref{m2}] 
By Corollary \ref{pbase} and Lemma \ref{wedgecap} we can describe ${\mathcal L}_m(s)$ as union of pyramid-abaci with bases $[1,s-1]$, $[s+1,2s-1]$, $[2s+1,3s-1],\cdots,[(m-1)s+1,ms-2]$. This uniquely determines the placement of beads in abacus positions each row $j$. In light of the structure described above, it is clear that the total number of spacers in the first $j+1$ rows (for a particular $j$) is given by 
\begin{eqnarray*}
\sum_{i=0}^{j} ((2i+1)m+1) 
&=& 
\frac{2mj(j+1)}{2} + (m(j+1) + (j+1)) \\
&=& 
m(j+1)^2 + (j+1)
\end{eqnarray*}
after elementary simplification.  

Thus, the contribution to the weight of this particular core at row $j$ is given by 
\begin{eqnarray}
\label{summation1}
&&
\sum_{\ell=0}^{m-2} ((mj^2 + j) + (j+1+\ell(2j+1))(s-2(j+1)) + \notag \\
&&
 \sum_{\ell=m-1}^{m-1} ((mj^2 + j) + (j+1+\ell(2j+1))(s-2(j+2)) \notag \\
&=& 
(m-1)(mj^2+2j+1)(s-(2j+1)) + (2j+1)(s-(2j+1))\left( \frac{(m-2)(m-1)}{2} \right) \notag \\
&& 
+ m(j+1)^2(s-(2j+2)) \notag \\
&=& \frac{(s-(2j+1))(m-1)m}{2}\left( 2j^2+2j+1 \right) + m(j+1)^2(s-(2j+2))
\end{eqnarray}
using elementary summation properties and straightforward algebraic simplifications.  

In order to determine the total weight of this core, we simply sum (\ref{summation1}) above over all relevant rows of the abacus.  This yields 
\begin{eqnarray*}
&& \sum_{j=0}^{t-2}  \frac{(s-(2j+1))(m-1)m}{2}\left( 2j^2+2j+1 \right) + m(j+1)^2(s-(2j+2)) \\
&=&  
ms(t-1)\left(\frac{m(t-1)^2}{3} + \frac{m}{6} +\frac{t-1}{2}\right) - m(t-1)^2\left( \frac{m(t-1)^2+1}{2} +t-1\right)
\end{eqnarray*}
using well--known results on sums of integer powers.  Replacing $s$ by $2t-1$ or $2t-2$ yields the results of this theorem after elementary simplification.  
\end{proof}
\vskip .2in 
The weight of a {\it maximal} $(s-1, s, s+1)$-core partitions was obtained by Amdeberhan-Leven (Theorem 4.3, \cite{AL}), Yang-Zhang-Zhou (Corollary 3.5, \cite{Y-Z-Z}) and Xiong (Corollary 1.2, \cite{X}). When $m=1$, the weight of a maximal $(s, ms-1,ms+1)$-core partition, agrees with the weight of the longest $(s, ms-1, ms+1)$.  This agreement leads us to the following conjecture.
\begin{conjecture}\label{Berger} The size of a maximal $(s,ms-1,ms+1)$-core is
\begin{enumerate}
\item $\frac{m^2t(t-1)(t^2-t+1)}{6}$ if $s=2t-1$\\
\item $\frac{m^2(t-1)^2(t^2-2t+3)}{6} - \frac{m(t-1)^2}{2}$ if $s=2t-2.$
\end{enumerate}
There are two such maximal partitions; one corresponding to ${\mathcal L}(s)$, and one corresponding to its conjugate.
\end{conjecture}
If Conjecture \ref{Berger} is true, then we have the following elegant corollary.
\begin{corollary} \label{3} Let $s$ be even. The weight of the maximal $(s,ms-1,ms+1)$-core partition is divisible by $m^2.$
\end{corollary}
\section{$ms$-abaci of self-conjugate $(s,ms\pm1)$-core partitions with distinct parts}
We close this paper by applying our tools to prove results on self-conjugate simultaneous core partitions with distinct parts.
The following is a well-known lemma.
\begin{lemma} \label{2core} The 2-core partitions are exactly those of the form $(k,k-1,k-2,\cdots,1)$. The bead-sets of the 2-cores are of the form $\{\cup_{\ell\leq 0}2\ell+1\}$.
\end{lemma}
\begin{lemma} \label{axis} Let $X$ be a bead set of a self-conjugate partition. Then there exists a half-integer $\theta$ such that if $x\in X$ and $x>\theta$ then there exists a $y\not\in X$ such that $|y-\theta|=|x-\theta|.$
\end{lemma}
\begin{proof} See Corollary 3.4 in \cite{R90}.
\end{proof}
\begin{lemma} \label{2conj} The self-conjugate partitions with distinct parts are exactly the $2$-core partitions.
\end{lemma}
\begin{proof} Every 2-core partition is clearly a self-conjugate partitions with distinct parts. Now suppose we have a self-conjugate partition $\lambda$ with distinct parts.  Then it must have a bead-set $X$ that consists of alternating spacer-and-beads.  Suppose not. If two beads occur in a row, we know that it violates having distinct parts. Suppose two spacers occur in a row. If $y,y+1\not\in X$ and both $y,y+1<\theta$ or both $y,y+1>\theta$ then, by Lemma \ref{axis}, there will be two beads in succession on the other side of $\theta$. If $y<\theta$ and $\theta<y+1$, then by Lemma \ref{axis} $\lambda$ is not self-conjugate.
\end{proof}
With the results of the previous sections and the lemmas above, we can consider self-conjugate simultaneous core partitions with distinct parts.
\begin{proposition} \label{Fstar} The number $F_*(s)$ of self-conjugate $(s,s+1)$-core partitions with distinct parts obeys the following relations:
$F_*(1)=1, F_*(2)=2$ and $F_*(2\alpha)=F_*(2\alpha+1)=\alpha+1$, where $\alpha\geq 1.$
\end{proposition}
\begin{proof} There is only one $(1,2)$-core, the empty partition. There are two $(2,3)$-cores, the empty partition and the partition $\lambda=(1).$ Suppose $n=2\alpha$. Then the self-conjugate $(2\alpha,(2\alpha)m+1)$-cores with distinct parts will be, by Lemma \ref{2conj}, the empty set plus the 2-cores that can be accommodated as sub-abaci of $\cup(i,0)$ where $1\leq i\leq 2\alpha-1.$ There are $\alpha$ such cores, and this number remains unchanged if $s=2\alpha+1.$
\end{proof}
\begin{proposition} \label{scEminus} The number $E^-_{m,*}(s)$ of self-conjugate $(s,ms-1)$-cores with distinct parts obeys the following relations:
$E^-_{m,*}(1)=1$, $E^-_{m,*}(2)=m$ and 
\begin{enumerate}
\item $E^-_{m,*}(2\alpha)=m\alpha$ and
\item $E^-_{m,*}(2\alpha+1)=\alpha+1$
\end{enumerate}
for all $m\geq 1$ and $\alpha \geq 1$.
\end{proposition}
\begin{proof} The argument is similar to Proposition \ref{Fstar}. There is only one $(1,m-1)$-core: the empty partition. There are $m$ self-conjugate $(2, 2m-1)$-cores, the empty set and the the partitions corresponding to $\cup^{m'}_{\ell=1}(2\ell-1,0)$, where $1\leq m'\leq m-1.$ This gives us the initial conditions. We consider separately the cases when $s$ is odd or even.
\begin{enumerate}
\item Suppose $s=2\alpha$, and $\alpha>0$. Then the self-conjugate $(2\alpha, (2\alpha)m-1)$-cores with distinct parts will be, by Lemma \ref{2conj} the empty set plus the 2-cores accommodated as sub-abaci of $\{\cup(i+(2\alpha)\ell,0)\cup(i'+(2\alpha)(m-1),0)\}$ as $0\leq i\leq 2\alpha-1$, $0\leq \ell\leq m-2$ and $0\leq i'\leq 2\alpha-2$. There are $m\alpha-1$ such 2-cores; when we count the empty partition we arrive at $m\alpha.$
\item Suppose $s=2\alpha+1$. Then the number of self-conjugate $(2\alpha,2\alpha+1)$-cores with distinct parts will be, by Lemma \ref{2conj}, the number of 2-cores that can be accommodated as sub-abaci of $\{\cup(i+(2\alpha+1)\ell,0)\cup(i'+(m-1)(2\alpha+1),0)\}$ as $1\leq i\leq 2\alpha$, $0\leq \ell\leq m-2$ and $1\leq i'\leq 2\alpha-1$. However, since $(2\alpha+1,0)\not\in {\mathcal E}^-_{m}(2\alpha+1)$, there are only $\alpha$ such non-empty 2-cores; those that can be accommodated from abacus positions $(i,0)$ where $1\leq i\leq 2\alpha.$
\end{enumerate}
\end{proof}
\begin{example} $F_*(8)=F_*(9)=5$. The set of self-conjugate $(8,9)$-core partitions with distinct parts is $\{\emptyset, (1), (2,1), (3,2,1), (4,3,2,1)\}$. This is also the set of self-conjugate $(9,10)$-core partitions with distinct parts. 
\end{example}
\begin{proposition} The number $E^+_{m,*}(s)$ of self-conjugate $(s,ms+1)$-cores obeys the following relations:
$E^+_{m,*}(1)=1$, $E^+_{m,*}(2)=m+1$ and 
\begin{enumerate}
\item $E^+_{m,*}(2\alpha)=m\alpha+1$ and
\item $E^+_{m,*}(2\alpha+1)=\alpha+1$
\end{enumerate}
for all $m\geq 1$ and $\alpha >1$.
\end{proposition}
\begin{proof} The argument in both cases is similar to ones above, with the added consideration that, for (1), the partition corresponding to $\{\cup^{\alpha}_{\gamma=1}\cup^{m-1}_{\ell=0}(2\gamma-1+\ell s,0)\}$ must also be counted.
\end{proof}
\begin{corollary} $E^-_{m,*}(2\alpha+1)=E^+_{m,*}(2\alpha+1).$
\end{corollary}
\begin{example} $E^-_{3,*}(5)=E^+_{3,*}(5)=3$. The set of self-conjugate $(5,14)$-cores with distinct parts is exactly $\{\emptyset, (1), (1,3)\}$, which is also the set of self-conjugate $(5,16)$-cores with distinct parts.
\end{example}
{\bf Acknowledgments} The first author was supported by PSC-CUNY Grant TRADA-46-493. The first author would like to thank Joe Gallian for the invitation to the University of Minnesota-Duluth in July 2016, where a portion of the manuscript was completed.  The authors have been in communication with Aaron Berger, who has made progress on Conjecture \ref{Berger}.
\newpage

\end{document}